\documentclass[12pt]{amsart}
\usepackage{epsfig}
\usepackage{amssymb, amsmath}
\usepackage{amsfonts,graphics,epsfig}
\usepackage{mathrsfs}
\usepackage{pb-diagram, pb-xy}
\usepackage[all]{xy}
\usepackage{comment}
\usepackage{tikz}
\usepackage[left=1.5in,right=1.5in,top=1.5in,bottom=1.5in]{geometry}
\usepackage[colorlinks=true, linkcolor=red, urlcolor=blue, citecolor=blue]{hyperref}

\newcommand{\R}{\mathbb{R}}

\newcommand{\Q}{\mathbb{Q}}
\def\P{\mathbb{P}}

\newcommand{\<}{\leq}
\def\>{\geq}
\def\e{\epsilon}

\def\subset{\subseteq}

\newcommand{\N}{\mathbb{N}}

\newcommand{\tbf}{\textbf}

\newcommand{\lrd}{\lfloor}
\newcommand{\rrd}{\rfloor}

\def\OO{\mathcal{O}}
\def\II{\mathscr{I}}
\def\mcB{\mathcal{B}}

\def\surjective{\twoheadrightarrow}

\newcommand{\blank}{\underline{\hskip 6pt}}
\newtheorem{theorem}{Theorem}[section]
\newtheorem{lemma}[theorem]{Lemma}
\newtheorem{proposition}[theorem]{Proposition}

\newtheorem*{maintheorem}{Theorem}
\theoremstyle{remark}
\newtheorem{remark}[theorem]{Remark}
\theoremstyle{definition}
\newtheorem{definition}[theorem]{Definition}
\theoremstyle{definition}

\numberwithin{equation}{section}

\def\div{\operatorname{div}}
\def\mod{\operatorname{mod}}
\def\deg{\operatorname{deg}}
\def\im{\operatorname{Im}}
\def\supp{\operatorname{Supp}}
\def\codim{\operatorname{codim}}
\def\chr{\operatorname{char}}
\def\ex{\operatorname{Ex}}
\def\lct{\operatorname{lct}}

\author{Omprokash Das}
\address{School of Mathematics\\
Tata Institute of Fundamental Research\\
1 Homi Bhabha Road,
Colaba, Mumbai 400005}
\email{omdas@math.tifr.res.in}
\address{Department of Mathematics\\
University of Utah\\
155 S 1400 E\\
Salt Lake City, Utah 84112}
\email{das@math.utah.edu}

\author{Christopher D. Hacon}
\address{Department of Mathematics\\
University of Utah\\
155 S 1400 E\\
Salt Lake City, Utah 84112}
\email{hacon@math.utah.edu}
\date{}

\begin{document}
\title{On the Adjunction Formula for $3$-folds in characteristic $p>5$}
%\subtitle{Adjunction formula for $3$-folds in $\mbox{Char }p>5$}
\maketitle

\begin{abstract}
In this article we prove a relative Kawamata-Viehweg vanishing-type theorem for PLT $3$-folds in characteristic $p>5$. We use this to prove the normality of minimal log canonical centers and the adjunction formula for codimension $2$ subvarieties on $\Q$-factorial $3$-folds in characteristic $p>5$.
\end{abstract}

\tableofcontents

%\keywords{Kawamata-Viehweg Vanishing Theorem in Positive Characteristic \and LC Centers \and Minimal LC Centers \and Adjunction Formula \and Subadjunction \and Canonical Bundle Formula \and Positive Characteristic.}

\section{Introduction} Let $(X,\Delta )$ be a log canonical pair and $W$ a minimal log canonical center, then (under mild technical assumptions) by Kawamata's celebrated subadjunction theorem, it is known that $W$ is normal and we can write $(K_X+\Delta )|_W =K_W+\Delta _W$ where $(W,\Delta _W)$ is Kawamata log terminal \cite{Kaw98} (see also \cite{Kaw97} and the references therein). The proof of this result is based on the Kawamata-Viehweg vanishing theorem and Hodge theory. 
These results are known to fail in characteristic $p>0$ and therefore one may expect that Kawamata's subadjunction also fails in this context. It should however be noted that related results have been obtained in the closely related context of $F$-singularities (see for example \cite{Sch09},  \cite{HX15} and \cite{Das15}) and that the minimal model program has been established for 3-folds in characteristic $p>5$  (see \cite{HX15} and \cite{Bir13}). In particular \cite{HX15} exploits the fact that PLT singularities in dimension 3 and characteristic $p>5$ are closely related to the analogous notion of purely $F$-regular singularities. In this paper, using the results from   \cite{HX15} and \cite{Bir13}, we show that in dimension $3$ and characteristic $p>5$ a relative version of the Kawamata-Viehweg vanishing theorem holds and we use this to establish that (under some mild technical conditions) the analog of 
Kawamata's subadjunction result holds. 

\begin{maintheorem}[Theorem \ref{thm:pl-vanishing}]
Let $f:(X, S+B\>0)\to Z$ be either a pl-divisorial contraction or a pl-flipping contraction (cf. Definition \ref{d-pl}) such that $\dim X =3$, $\chr p>5$ and $S=\lfloor S+B\rfloor$ is reduced and irreducible. If the maximum dimension of the fibers of $f$ is $1$, then $R^if_*\OO_X(-S)=0$ for all $i>0$.	
\end{maintheorem}

This result allows us to prove the normality of the minimal LC centers for $3$-folds.
\begin{maintheorem}[Theorem \ref{normality}, \ref{adjunction}]
Let $(X, \Delta)$ be a $\Q$-factorial $3$-fold log canonical pair such that $X$ has Kawamata Log Terminal singularities. If $W$ is a minimal log canonical center of $(X, \Delta)$, then $W$ is normal. If moreover the coefficients of $\Delta$ belong to a DCC set $I\subset [0,1]$ and char $k>{\rm max}\{5, \frac{2}{\delta}\}$, where $\delta>0$ is the minimum of the set $D(I)\cap (0,1]$ (where $D(I)$ is defined in \ref{dcc-def}), then the following hold:
\begin{enumerate}
\item There exists effective $\Q$-divisors $\Delta _W$ and $M_W$ on $W$ such that $(K_X+\Delta )|_W\sim_{\Q} K_W+\Delta _W+M_W$. Moreover, if $\Delta=\Delta '+\Delta ''$ with $\Delta '$ (resp. $\Delta ''$) the sum of all irreducible components which contain (resp. do not contain) $W$, then $M_W$ is determined only by the pair $(X, \Delta ')$.
\item There exists an effective $\Q$-divisor $M'_W$ such that $M'_W\sim_{\Q} M_W$ and the pair $(W, \Delta _W+M'_W)$ is KLT.
\end{enumerate} 
\end{maintheorem}

All of the results in this article hold in characteristic $p>5$ unless stated otherwise. We will use the standard terminologies and notations from \cite{KM98}. We also use the abbreviations: LC for log canonical, KLT for Kawamata log terminal, PLT for purely log terminal, DLT for divisorially log terminal, NLC for non-log canonical centers, NKLT centers for non-Kawamata log terminal centers, and $\lct$ for log canonical thresholds. If $(X,\Delta )$ is LC, then the NKLT centers are also known as log canonical centers or LC centers.\\

{\bf Acknowledgements.} 
	This research  was partially supported by the National Science Foundation research grants No: DMS-$1300750$, DMS-$1265285$, FRG grant No. DMS-$1265261$, and Simons Grant Award No. $256202$.\\

\section{Properties of Log Canonical Centers}
In this section we establish some basic properties of LC centers.\\
%%%%%%%%%%%%%%%%%%%%
\begin{lemma}\label{lcc-intersection}
Let $X$ be a $\Q$-factorial KLT $3$-fold and $(X, \Delta\>0)$ a log canonical pair. Let $W_1$ and $W_2$ be two log canonical centers of $(X, \Delta)$. Then every irreducible component of $W_1\cap W_2$ is a log canonical center of $(X, \Delta)$. 	
\end{lemma}
%%%%%%%%%%%%%%%%%%
\begin{proof}
There are three cases depending on the codimension of $W_1$ and $W_2$.\\

\tbf{Case I}: \textit{ $\codim _X W_1=\codim _X W_2=1$}. In this case $W_1$ and $W_2$ are components of $\Delta$. Let $\Delta=W_1+W_2+\bar{\Delta}$. Then by adjunction we have $$(K_X+W_1+W_2+\bar{\Delta})|_{W_1^n}=K_{W_1^n}+\text{ Diff}_{W_1^n}(\bar{\Delta})+W_2|_{W_1^n},$$ where $W_1^n\to W_1$ is the normalization. By localizing at the generic point of an irreducible component of $W_1\cap W_2$ we reduce to a surface problem. Now, on a surface in characteristic $p>0$, the relative Kawamata-Viehweg vanishing and Koll\'ar's connectedness theorem hold (see \cite[10.13]{Kol13} and \cite[3.1]{Das15}). Thus on a surface the intersection of two LC centers is a LC center and we are done by the usual argument (cf. \cite[Proposition 1.5]{Kaw95}).\\

\tbf{Case II}: \textit{$\codim _X W_1=1$ and $\codim _XW_2=2$}. Since $X$ is $\Q$-factorial, $(X, (1-\e)\Delta)$ is KLT for any $0<\e<1$. Thus by \cite[7.7]{Bir13}, there exists a $\Q$-factorial model $f': X'\to X$ of relative Picard number $\rho(X'/X)=1$ such that $\mbox{Ex}(f')$ is the unique exceptional divisor $E'$ with center $W_2$, and
\begin{equation}\label{X'-intersection}
K_{X'}+E'+W_1'+\Delta'=f'^*(K_X+\Delta),
\end{equation} 
where $\Delta'\>0$, and $W_1'$ is the strict transform of $W_1$ under $f'$.\\

Since $W_1'$ and $E'$ are $\Q$-Cartier, they intersect along a curve (possibly reducible). Let $C'$ be an irreducible component of $W_1'\cap E'$. Then by Case I, $C'$ is a LC center of $(X', E'+W_1'+\Delta'\>0)$. Since every irreducible component of $W_1\cap W_2$ is dominated by an irreducible component of $W_1'\cap E'$, we are done by relation \eqref{X'-intersection}. \\

\tbf{Case III}: \textit{$\codim _XW_1=\codim _XW_2=2$}. Again, since $X$ is $\Q$-factorial, $(X, (1-\e)\Delta)$ is KLT for any $0<\e<1$. Thus by \cite[7.7]{Bir13}, there exists a $\Q$-factorial model $f': X'\to X$ extracting two exceptional divisors (one at a time) $E_1'$ and $ E_2'$ such that $E_1'\cap E_2'\neq \emptyset, f'(E_1')=W_1$ and $f'(E_2')=W_2$, and 
\begin{equation}\label{X'-2intersection}
K_{X'}+E_1'+E_2'+\Delta'=f'^*(K_X+\Delta).
\end{equation}
Since $E_1'$ and $E_2'$ are $\Q$-Cartier, they intersect along a curve (possibly reducible). Let $C'$ be an irreducible component of $E_1'\cap E_2'$. Then by Case I,  $C'$ is a LC center of $(X', E_1'+E_2'+\Delta'\>0)$. Since every irreducible component of $W_1\cap W_2$ is dominated by an irreducible component of $E_1'\cap E_2'$, we are done by relation \eqref{X'-2intersection}.
\end{proof}
%%%%%%%%%%%%%%%%%%%%%

%%%%%%%%%%%%%%%%%%%%%%%%%
The following proposition is a characteristic $p>5$ version of Fujino's adjunction theorem for DLT pairs (see \cite[3.9.2]{Cor07} and \cite[4.16]{Kol13}) on a $\Q$-factorial $3$-fold.\\

\begin{proposition}[DLT Adjunction]\label{dlt-adjunction}
Let $(X, \Delta\>0)$ be a $\Q$-factorial DLT $n$-fold with $n\leq 3$ such that $\Delta=D_1+D_2+\cdots+D_r+B$ and $\lrd\Delta\rrd=D_1+D_2+\cdots+D_r$, where the $D_i$'s are prime divisors. Assume that $\chr p>5$. Then the following hold:
\begin{enumerate}
\item The $s$-codimensional log canonical centers of $(X, \Delta)$ are exactly the irreducible components of the various intersections $D_{i_1}\cap\cdots\cap D_{i_s}$ for some $\{i_1,\dots, i_s\}\subset \{1,\dots, r\}$.
\item Every irreducible component of $D_{i_1}\cap\cdots\cap D_{i_s}$ is normal and of pure codimension $s$.
\item Let $W$ be a log canonical center of $(X, \Delta)$, then there exists an effective $\Q$-divisor $\Delta_W\>0$ on $W$ such that $(K_X+\Delta)|_W\sim_\Q K_W+\Delta_W$ and $(W, \Delta_W)$ is DLT.
\item If $D_i\cap D_j=\emptyset$ for all $i\neq j$, then $(X, \Delta)$ is in fact PLT.	
\end{enumerate}	
\end{proposition}
%%%%%%%%%%%%%%%%%%%%%%%

\begin{proof} The result is well known in dimension $\leq 2$.
(1) follows from the proof in \cite[Theorem 4.16]{Kol13}.

Since $X$ is $\Q$-factorial, $(X, D_i)$ is also PLT and then by adjunction $(D^n_i, \text{Diff}_{D^n_i})$ is KLT, where $D^n_i\to D_i$ is the normalization. Since $\mbox{Diff}_{D^n_i}$ has standard coefficients, by \cite{Har98} and \cite[3.1]{HX15},  $(D^n_i, \text{Diff}_{D^n_i})$ is strongly $F$-regular in characteristic $p>5$. Then by \cite[4.1]{HX15} and \cite[4.1, 5.4]{Das15}, $D_i$ is normal. This proves that every irreducible component of $\lrd \Delta\rrd$ is normal and hence $(2)$ holds for $s=1$.\\

It is easy to see that $(D_i, \text{Diff}_{D_i}(\Delta-D_i))$ 
 is DLT, and so $D_i$ is a $\Q$-factorial surface by \cite[6.3]{FT12}. (2) and (3) now follow from the result in dimension $2$. (4) is immediate.
\end{proof}
%%%%%%%%%%%%%%%%%%%%

\section{Vanishing Theorem and Minimal Log Canonical Centers}
In this section we will prove a relative vanishing theorem and then use it to prove the normality of minimal log canonical centers.\\
%%%%%%%%%%%%%%%%%%%%%

%%%%%%%%%%%%%
\begin{definition}\label{d-pl}
	Let $f:X\to Z$ be a projective birational morphism between normal quasi-projective varieties with relative Picard number $\rho(X/Z)=1$. Let $(X, S+B\>0)$ be a $\Q$-factorial PLT pair such $\lrd S+B\rrd=S$ is irreducible, and $-S$ and $-(K_X+S+B)$ are both $f$-ample. 
\begin{enumerate}
	\item If $\dim \ex(f)=\dim X-1$, then $f:X\to Z$ is called a \emph{pl-divisorial contraction}.
	\item If $\dim \ex(f)<\dim X-1$, then $f:X\to Z$ is called a \emph{pl-flipping contraction}.
	\end{enumerate}	 
	\end{definition}
%%%%%%%%%%%%%%%%%%%%%%%%%

\begin{proposition}\label{prop:surjective}
Let $(X, S+B)$ be a $\Q$-factorial $3$-fold PLT pair, where $S$ is a prime Weil divisor. Assume that $(p^e-1)(K_X+S+B)$ is an integral Weil divisor for some $e>0$. Then there exists an integer $e_0\gg 0$ such that the following sequence
\begin{equation}\label{seq:e_01}
	\xymatrix{
	0\ar[r]&  \mathcal{B}_{ne_0}\ar[r]& F^{ne_0}_*\OO_X((1-p^{ne_0})(K_X+B)-p^{ne_0}S)\ar[r]^-{\phi_{ne_0}}& \OO_X(-S)\ar[r]& 0}
	\end{equation}
%\begin{equation}\label{seq:e_02}
%	\begingroup
%	\fontsize{9pt}{12pt}\selectfont
%\xymatrixcolsep{0.5cm}	\xymatrix{F^{ne_0}_*\OO_X((1-p^{ne_0})(K_X+B)-p^{ne_0}S)\ar[r] & F^{(n-1)e_0}_*\OO_X((1-p^{(n-1)e_0})(K_X+B)-p^{(n-1)e_0}S)\ar[r] & 0, }
%	\endgroup
%	\end{equation}
is exact at all codimension $2$ points of $X$ contained in $S$,	for all $n\>1$, where $\phi_{ne_0}$ is defined by the trace map (see \cite{Sch14} and \cite{Pat14}) and $\mcB_{ne_0}$ is the kernel of $\phi_{ne_0}$.
\end{proposition}

\begin{proof}
By Proposition \ref{dlt-adjunction}, $S$ is normal. Since the question is local on $X$, we may assume that $X$ is affine. Then by \cite[2.13]{HX15}, we can choose an effective $\Q$-Cartier divisor $G\>0$ not containing $S$ and with sufficiently small coefficients such that $K_X+S+B+G$ is $\Q$-Cartier with index not divisible by $p$.\\

Localizing $X$ at a codimension $2$ point of $X$ contained in $S$, we may assume that $X$ is an excellent surface. Then by adjunction we have $(K_X+S+B+G)|_S=K_S+B_S+G|_S$, where $B_S$ is the Different. Since $(X, S+B)$ is PLT, $(S, B_S)$ is KLT by adjunction. Now, $(S, B_S)$ is strongly $F$-regular by \cite[2.2]{HX15}, since $S$ is a smooth curve. Since the coefficients of $G$ are sufficiently small, $(S, B_S+G|_S)$ is also strongly $F$-regular. Therefore we get the following surjection
		\begin{equation*}
			F^e_*\OO_S((1-p^e)(K_S+B_S+G|_S))\surjective \OO_S,
		\end{equation*}
		for all $e\gg 0$ and sufficiently divisible.\\

		We have the following commutative diagram
		%\begingroup
		%\fontsize{10pt}{12pt}\selectfont
		\begin{equation}\label{diagram-1-pl}
			\xymatrix{ F^e_*\OO_X((1-p^e)(K_X+S+B+G))\ar[d]\ar@{->>}[r]& F^e_*\OO_S((1-p^e)(K_S+B_S+G|_S))\ar@{->>}[d]\\
			\OO_X\ar@{->>}[r]&  \OO_S }
		\end{equation}
		%\endgroup
		To see the surjectivity of the top arrow note that since $F_*^e$ is exact, it suffices to show that $(1-p^e)(K_X+S+B+G)|_S=(1-p^e)(K_S+B_S+G|_S)$, and since  $(1-p^e)(K_X+S+B+G)$ and $(1-p^e)(K_S+B_S+G|_S)$ are Cartier for $e\gg 0$, it suffices to show that this equality holds at codimension $1$ points of $S$, but this is clear since  $(K_X+S+B+G)|_S=K_S+B_S+G|_S$.
		Since the ring $\OO_X$ is local, the surjectivity of the second vertical map (along with Nakayama's Lemma) implies the surjectivity of the first vertical map, i.e.,
		\begin{equation}\label{OX-onto-pl}
			\xymatrix{F^{ne_0}_*\OO_X((1-p^{ne_0})(K_X+S+B+G))\ar@{->>}[r]& \OO_X }
		\end{equation}
		is surjective for all $n\>1$, where $e_0\gg 0$ is sufficiently divisible.\\
		
		Since the map \eqref{OX-onto-pl} factors through $F^{ne_0}_*\OO_X((1-p^{ne_0})(K_X+B))$, we get the following surjectivity
		\begin{equation}\label{KY-onto-pl}
			\xymatrix{F^{ne_0}_*\OO_X((1-p^{ne_0})(K_X+B))\ar@{->>}[r]^-{\psi_{ne_0}}& \OO_X. }
		\end{equation}
		Let $s$ be a pre-image of $1$ under $\psi_{ne_0}$, then we get the following splitting of $\psi_{ne_0}$
		 \begin{equation}\label{Psi-splitting-pl}
		\xymatrix{\OO_X\ar[r]^-{\blank \cdot s} & F^{ne_0}_*\OO_X((1-p^{ne_0})(K_X+B))\ar[r]^-{\psi_{ne_0}}& \OO_X. }
		\end{equation}
		Twisting \eqref{Psi-splitting-pl} by $\OO_X(-S)$ and taking reflexive hulls we get the following splitting
		\begin{equation}\label{S-splitting-pl}
			\xymatrix{\OO_X(-S)\ar[r]& F^{ne_0}_*\OO_X((1-p^{ne_0})(K_X+B)-p^{ne_0}S)\ar[r]& \OO_X(-S).}
		\end{equation}
		In particular the morphism 
		\begin{equation}\label{final-onto-pl}
			\xymatrix{F^{ne_0}_*\OO_X((1-p^{ne_0})(K_X+B)-p^{ne_0}S)\ar@{->>}[r]& \OO_X(-S)} 
		\end{equation}
	is surjective for all $n\>1$.\\

%Now consider the sequence \eqref{S-splitting-pl} for $n=1$. Then twisting \eqref{S-splitting-pl} (for $n=1$) by $\OO_X((1-p^{(n-1)e_0})(K_X+S+B))$ and taking the reflexive hulls and pushing forward by $F^{(n-1)e_0}_*$ we get the following splitting
%\begin{equation}\label{seq:S-e0-splitting}
%	\begingroup
%	\fontsize{9pt}{12pt}\selectfont
%\xymatrix{	F^{(n-1)e_0}_*\OO_X((1-p^{(n-1)e_0})(K_X+B)-p^{(n-1)e_0}S)\ar[r]& F^{ne_0}_*\OO_X((1-p^{ne_0})(K_X+B)-p^{ne_0}S)\ar[dl]\\ 
 %F^{(n-1)e_0}_*\OO_X((1-p^{(n-1)e_0})(K_X+B)-p^{(n-1)e_0}S) & }
%	\endgroup
%	\end{equation}
%Therefore we get the following surjectivity
%\begin{equation}\label{seq:S-e0-surjective}	
%	\begingroup
%	\fontsize{9pt}{12pt}\selectfont
%\xymatrix{F^{ne_0}_*\OO_X((1-p^{ne_0})(K_X+B)-p^{ne_0}S)\ar@{->>}[r]& F^{(n-1)e_0}_*\OO_X((1-p^{(n-1)e_0})(K_X+B)-p^{(n-1)e_0}S).}
%	\endgroup
%\end{equation}	

%Thus by induction we get the required surjectivity. 
%\begin{equation}\label{seq:S-ne0-surjective}
%	\begingroup
%	\fontsize{9pt}{12pt}\selectfont	
%\xymatrix{F^{ne_0}_*\OO_X((1-p^{ne_0})(K_X+B)-p^{ne_0}S)\ar@{->>}[r]& F^{(n-1)e_0}_*\OO_X((1-p^{(n-1)e_0})(K_X+B)-p^{(n-1)e_0}S).}
%	\endgroup
%\end{equation}
	
	\end{proof}
%%%%%%%%%%%%%%%%%%%%%%%%%%%%%%%
%%%%%%%%%%%%%%%%%%%%%%%%%%%%%%%
\begin{remark}
	In Proposition \ref{prop:surjective}, if we further assume that the coefficients of $B$ are in the standard set $I=\{1-\frac{1}{n}: n\>1\}$, then it follows that the sequence \eqref{seq:e_01} is exact at all codimension $2$ points of $X$. Indeed, by localizing at a codimension $2$ point $P\in X\setminus S$, we may assume that $(X, B)$ is an excellent surface. In this case the RHS of the sequence \eqref{seq:e_01} takes the following form
\begin{equation}\label{seq:not-S}
	\xymatrixcolsep{3pc}\xymatrix{F^{ne_0}_*\OO_X((1-p^{ne_0})(K_X+B))\ar[r]^-{\phi_{ne_0}}& \OO_X.}
	\end{equation} 
Since $(X, B)$ is a PLT surface and $\lrd B\rrd=0$, $(X, B)$ is KLT. Thus by \cite{Har98} and \cite[3.1]{HX15}, $(X, B)$ is strongly $F$-regular in $\chr\ p>5$. Since the coefficients of $G$ are sufficiently small, $(X, B+G)$ is also strongly $F$-regular. Therefore we get the following surjectivity
\begin{equation}\label{seq:not-S-sur}
\xymatrixcolsep{3pc}\xymatrix{F^{ne_0}_*\OO_X((1-p^{ne_0})(K_X+B+G))\ar@{->>}[r]& \OO_X,}
	\end{equation}
for some $e_0\gg 0$ and sufficiently divisible, and for all $n\>1$.\\
Since the map in \eqref{seq:not-S-sur} factors through $F^{ne_0}_*\OO_X((1-p^e)(K_X+B))$, we get the following surjectivity
\begin{equation}
\xymatrixcolsep{3pc}\xymatrix{F^{ne_0}_*\OO_X((1-p^{ne_0})(K_X+B))\ar@{->>}[r]^-{\phi_{ne_0}}& \OO_X.}
\end{equation} 
\end{remark}	

\begin{remark}
	In Proposition \ref{prop:surjective}, we can further show that the sequence \eqref{seq:e_01} is exact at all codimension $3$ points of $X$ contained in $S$, if the coefficients of $B$ are in the standard set $I=\{1-\frac{1}{n}: n\>1\}$. Indeed, by localizing at a codimension $3$ point of $X$ contained in $S$, we may assume that $X$ is an excellent $3$-fold. Then $(S, B_S)$ is a KLT surface pair. By \cite{Har98} and \cite[3.1]{HX15}, $(S, B_S)$ is strongly $F$-regular in $\chr\ p>5$. The rest of the proof runs without any changes.\\
	\end{remark}

%%%%%%%%%%%%%%%%%%%%%%%%%%%%%%%
%%%%%%%%%%%%%%%%%%%%%%%%%%%%%%%

%%%%%%%%%%%%%%%%%%%%%%%%%%%%%%%%%%%%%%
\begin{theorem}\label{thm:pl-vanishing}
	Let $f:(X, S+B\>0)\to Z$ be either a pl-divisorial contraction or a pl-flipping contraction. If the maximum dimension of the fibers of $f$ is $1$, then $R^if_*\OO_X(-S)=0$ for all $i>0$.
	\end{theorem}
%%%%%%%%%%%%%%%%%%%%%%%%%%%%%%%%%%%%%%
\begin{proof} Since $X$ is $\Q$-factorial, by perturbing the coefficients of $B$ we may assume that $(p^e-1)(K_X+S+B)$ is an integral Weil divisor for some $e>0$. 
	Since $f$ is birational and $\ex(f)\subset \supp(S)$, it is enough to show that $R^if_*\OO_X(-S)=0$ in a neighborhood of $f(S)$. Thus by restricting $(X, S+B)$ on a suitable neighborhood of $S$ and by Proposition \ref{prop:surjective}, we may assume that the following sequence is exact at all codimension $2$ points of $X$
	\begin{equation}\label{S-ses-pl}
		\xymatrix{
		0\ar[r]&  \mathcal{B}_e\ar[r]& F^e_*\OO_X((1-p^e)(K_X+B)-p^eS)\ar[r]^-{\phi_e}& \OO_X(-S)\ar[r]& 0,}
	\end{equation}
	for all $e\gg 0$ and sufficiently divisible.\\

 The sequence \eqref{S-ses-pl} can be split into the following two exact sequences 
	\begin{equation}\label{split-1-pl}
		\xymatrix{0\ar[r]&  \mathcal{B}_e\ar[r]& F^e_*\OO_X((1-p^e)(K_X+B)-p^eS)\ar[r]^-{\phi_e}\ar[r]& \im(\phi_e)\ar[r]& 0}
	\end{equation}  
	\[\mbox{and} \]
	\begin{equation}\label{split-2-pl}
		\xymatrix{0\ar[r]& \im(\phi_e)\ar[r]& \OO_X(-S)\ar[r]& \mathcal{Q}_e\ar[r]& 0,}
	\end{equation}
	where $\mathcal{Q}_e$ is the corresponding quotient.\\

	Pushing forward the exact sequence \eqref{split-1-pl} by $f_*$ we get
	\begin{equation}\label{zero-1-pl}
		R^if_*(F^e_*\OO_X((1-p^e)(K_X+B)-p^eS))\to R^if_*\im(\phi_e) \to R^{i+1}f_*\mathcal{B}_e.
	\end{equation}
	Now $R^{i+1}f_*\mathcal{B}_e=0$ for all $i>0$, since the maximum dimension of the fiber of $f$ is $1$.\\	
Let $r$ be the index of $K_X+S+B$ and $H=-(K_X+S+B)$. By the division algorithm, there exist integers $k\>0$ and $0\<b<r$ such that $(p^e-1)=r\cdot k+b$. Then by Serre vanishing 
	\[R^if_*(F^e_*\OO_X((1-p^e)(K_X+B)-p^eS))=F^e_*(R^if_*\OO_X(k\cdot rH-b(K_X+S+B)-S))=0, \] 
	for all $e\gg 0$ and sufficiently divisible, and $i>0$, since $H$ is $f$-ample.\\

	Thus from \eqref{zero-1-pl} we get
	\begin{equation}\label{Im-0-pl}
		R^if_*\im(\phi_e)=0,
	\end{equation}
for all $i>0$.\\

	Again, pushing forward the exact sequence \eqref{split-2-pl}  by $f_*$ we get
	\begin{equation}\label{zero-2-pl}
		R^if_*\im(\phi_e)\to R^if_*\OO_X(-S)\to R^if_*\mathcal{Q}_e.
	\end{equation}
	$R^if_*\mathcal{Q}_e=0$ for all $i>0$, since $\mathcal{Q}_e$ is supported at finitely many points, by \eqref{S-ses-pl}. Thus we have
	\begin{equation}\label{vanishing-proved-d}
	R^if_*\OO_X(-S)=0,
	\end{equation}
for all $i>0$.
	\end{proof}

%%%%%%%%%%%%%%%%%%%%%%%%%%%%%%%%%%%%%%%%
%%%%%%%%%%%%%%%%%%%%%%%%%%%%%%%%%%%%%%%
%%%%%%%%%%%%%%%%%%%%%%%%%%%%%%
\begin{theorem}\label{normality}
Let $(X, \Delta)$ be a $\Q$-factorial $3$-fold log canonical pair such that $X$ has KLT singularities. If $W$ is a minimal log canonical center of $(X, \Delta)$, then $W$ is normal.  
\end{theorem}

%proof%%%%%%%%%%%%%%%%%%%%%
\begin{proof}
Since $X$ is $\Q$-factorial and KLT, $(X, (1-\e)\Delta)$ is KLT for any $0<\e<1$, and all log canonical centers of $(X, \Delta)$ are contained in $\Delta$. Then by Reid's Tie Breaking trick (see \cite[8.7.1]{Cor07}) we may assume that $W$ is the unique log canonical center of $(X, \Delta)$ with a unique divisor over $X$ of discrepancy $-1$. There are two cases depending on the codimension of $W$.\\

\textbf{Case I:} \textit{ $\codim_X( W)=1$}.
Since $X$ is $\Q$-factorial, $(X, W)$ is log canonical. By adjunction $(K_X+W)|_{W^n}=K_{W^n}+\text{ Diff}_{W^n}$, where $W^n\to W$ is the normalization and $(W^n, \text{ Diff}_{W^n})$ is KLT. Thus by \cite{Har98} and \cite[3.1]{HX15},  $(W^n, \text{ Diff}_{W^n})$ is strongly $F$-regular in characteristic $p>5$. Then $W^n=W$, i.e., $W$ is normal by \cite[4.1]{HX15} or \cite[4.1]{Das15}.\\

\textbf{Case II:} \textit{ $\codim_X( W)=2$}.
Let $f:(Y, S+\Delta')\to (X, \Delta)$ be an extraction of the unique exceptional divisor $S$ over $X$ such that 
\begin{equation*}
K_Y+S+\Delta'=f^*(K_X+\Delta).
\end{equation*}
Note that $-S$ is $f$-ample. Since $(Y, S+\Delta')$ is PLT, $S$ is normal by Proposition \ref{dlt-adjunction}. Also, since $Y$ is $\Q$-factorial, $(Y, S)$ is PLT.\\

Consider the following exact sequence
\begin{equation*}
	\xymatrix{0\ar[r]& \OO_Y(-S)\ar[r] & \OO_Y\ar[r] & \OO_S\ar[r]& 0.}
\end{equation*}

Since $W$ is contained in the support of $\Delta$, $\Delta'\cap S\neq \emptyset$, and hence $-(K_Y+S)$ is $f$-ample. Thus $f: (Y, S)\to X$ is a pl-divisorial contraction. Then by Theorem \ref{thm:pl-vanishing}, $R^1f_*\OO_Y(-S)=0$, and we get the following exact sequence 
\begin{equation}\label{codim2mainses}
	\xymatrix{0\ar[r] & f_*\OO_Y(-S)\ar[r] & f_*\OO_Y\ar[r] & f_*\OO_S\ar[r] & 0. }
\end{equation}

Since $f_*\OO_Y(-S)=\II_W$ and $f_*\OO_Y=\OO_X$, we get 
\begin{equation}\label{codim2finalses}
	\xymatrix{0\ar[r] & \II_W\ar[r] & \OO_X\ar[r] & f_*\OO_S\ar[r] & 0. }
\end{equation}

Now $\OO_X\surjective f_*\OO_S$ factors in the following way

\begin{equation}
\xymatrix{ & \OO_W\ar@{^{(}->}[rd]\\
& & \nu_*\OO_{W^n}\ar@{^{(}->}[rd]\\
\OO_X\ar@{->>}[ruu]\ar@{->>}[rrr] & & & f_*\OO_S
}	
\end{equation}
where $\nu:W^n\to W$ is the normalization morphism.\\

Hence $\OO_W=\nu_*\OO_{W^n}$, i.e. $W$ is normal. 
\end{proof}
%%%%%%%%%%%%%%%%%%%%%%%%%%%%%%%%

\section{Adjunction Formula}
In this section we will prove an adjunction formula for $3$-folds in characteristic $p>5$. To start with we will need the following definitions and results.\\

\begin{definition}[DCC sets]\label{dcc-def}
We say that a set $I$ of real numbers satisfies the \emph{descending chain condition} or DCC, if it does not contain any infinite strictly decreasing sequence. For example,
\[I=\left\{\frac{r-1}{r}: r\in \N \right\} \]
satisfies the DCC.\\

Let $I\subset [0, 1]$. We define 
\[I_+:=\{j\in [0, 1]: j=\sum_{p=1}^l i_p \mbox{ for some } i_1, i_2,\dots, i_l\in I\} \]
and
\[D(I):=\{a\<1: a=\frac{m-1+f}{m}, m\in \N, f\in I_+  \}. \]
\end{definition}
%%%%%%%%%%%%%%%%%%%%%%%%%%

\begin{lemma}\cite[4.4]{MP04}\label{dcc-lemma}
	Let $I\subset [0, 1]$. Then
	\begin{enumerate}
		\item $D(D(I))=D(I)\cup \{1\}$.
		\item $I$ satisfies DCC if and only if $\bar{I}$ satisfies the DCC, where $\bar{I}$ is the closure of $I$.
		\item $I$ satisfies DCC if and only if $D(I)$ satisfies the DCC. 
		\end{enumerate}
	\end{lemma}

\begin{lemma}\label{dcc-adjunction}\cite[Lemma 2.3]{CGS14}\cite[Lemma 4.3]{MP04}\cite[Lemma 4.1]{HMX14} Let $(X, \Delta\>0)$ be a log canonical pair such that the coefficients of $\Delta$ belong to a set $I\subset [0, 1]$. Let $S$ be a normal irreducible component of $\lrd \Delta\rrd$ and $\Theta\>0$ be the $\Q$-divisor on $S$ defined by adjunction:
	\[(K_X+\Delta)|_S=K_S+\Theta. \]
Then, the coefficients of $\Theta$ belong to $D(I)$.	
	\end{lemma}	
	
%%%%%%%%%%%%%%%%%%%%%
\begin{definition}[Divisorial part and Moduli part]\cite[Section 7]{PS09}\cite[Section 6]{CTX13}\label{boundary-def}
	Let $f: X\to Z$ be a surjective proper morphism between two normal varieties and $K_X+D\sim_\Q f^*L$, where $D$ is a boundary divisor on $X$ and $L$ is a $\Q$-Cartier $\Q$-divisor on $Z$. Let $(X, D)$ be LC near the generic fiber of $f$, i.e., $(f^{-1}U, D|_{f^{-1}U})$ is LC for some Zariski dense open subset $U\subset Z$. Then we define two divisors $D_{\div}$ and $D_{\mod}$ on $Z$ in the following way:
\begin{align*}
D_{\div} &=\sum (1-c_Q)Q, \text{ where } Q\subset Z \text{ are prime Weil divisors of } Z,\\
c_Q &=\text{sup}\{ c\in \R : (X, D+cf^*Q) \text{ is LC over the generic point } \eta_Q \text{ of } Q \} \text{ and }\\
D_{\mod} &=L-K_Z-D_{\div}, \text{ so that } K_X+D\sim_\Q f^*(K_Z+D_{\div}+D_{\mod}).
\end{align*}
\end{definition}

\begin{remark}
Observe that $D_{\div}$ is a fixed divisor on $Z$, called the \emph{Divisorial part} and $D_{\mod}$ is a $\Q$-linear equivalence class on $Z$, called the \emph{Moduli part}. For other properties of $D_{\div}$ and $D_{\mod}$ see \cite[Section 7]{PS09} and \cite[Section 3]{Amb99}.
	\end{remark}

Let $\overline{\mathcal{M}}_{0, n}$ be the moduli space of $n$-pointed stable curves of genus $0$, $f_{0, n}: \overline{\mathcal{U}}_{0, n}\to \overline{\mathcal{M}}_{0, n}$ the universal family, and $\mathcal{P}_1, \mathcal{P}_2,\cdots , \mathcal{P}_n$, the sections of $f_{0, n}$ which correspond to the marked points (see \cite{Kee92} and \cite{Knu83}). Let $d_j \ (j=1,2,\cdots, n)$ be rational numbers such that $0<d_j\<1$ for all $j$, $\sum_j d_j=2$ and $\mathcal{D}=\sum_{j}d_j\mathcal{P}_j$.

%%%%%%%%%%%%%%%%%%%%%%%%%
\begin{lemma}\label{L-semi-ample}
\begin{enumerate}
\item There exists a smooth projective variety $\mathcal{U}^*_{0, n}$, a $\P^1$-bundle $g_{0, n}: \mathcal{U}^*_{0, n}\to \overline{\mathcal{M}}_{0, n}$, and a sequence of blowups with smooth centers 
\begin{equation*}
\xymatrix{	\overline{\mathcal{U}}_{0, n}=\mathcal{U}^{(1)}\ar[r]^-{\sigma_2} & \mathcal{U}^{(2)}\ar[r]^-{\sigma_3}& \cdots \ar[r]^-{\sigma_{n-2}}& \mathcal{U}^{(n-2)}=\mathcal{U}^*_{0, n}. }	
\end{equation*}

\item Let $\sigma: \overline{\mathcal{U}}_{0, n}\to \mathcal{U}^*_{0, n}$ be the induced morphism, and $\mathcal{D}^*=\sigma_*\mathcal{D}$. Then $K_{\overline{\mathcal{U}}_{0, n}}+\mathcal{D}-\sigma^*(K_{\mathcal{U}^*_{0, n}}+\mathcal{D}^*)$ is effective.

\item There exists a semi-ample $\Q$-divisor $\mathcal{L}$ on $\overline{\mathcal{M}}_{0, n}$ such that 
\[K_{\mathcal{U}^*_{0, n}}+\mathcal{D}^*\sim_{\Q} g^*_{0, n}(K_{\overline{\mathcal{M}}_{0, n}}+\mathcal{L}).
 \]
\end{enumerate}
\end{lemma}

\begin{proof}
The proof in \cite[Theorem 2]{Kaw97} works in positive characteristic without any change (see also \cite[6.7]{CTX13}, \cite[8.5]{PS09} and \cite[Section 3]{KMM94}).
\end{proof}
%%%%%%%%%%%%%%%%%

%%%%%%%%%%%%%%%%%
\begin{lemma}[Stable Reduction Lemma]\label{semistable-reduction}
Let $B$ be a smooth curve and $f:X\to B$, a flat family of rational curves such that the general fiber is isomorphic to $\P^1$, and a unique singular fiber $X_0$ over $0\in B$. Also assume that $ f|_{X^*}: (X^*=X\backslash X_0; \mathcal{P}_1, \mathcal{P}_2,\dots, \mathcal{P}_n) \to B^*=B-\{0\}$ is a flat family of $n$-pointed stable rational curves sitting in the following commutative diagram
\begin{equation}\label{canonical-diagram}
	\xymatrixcolsep{5pc}\xymatrix{ X^*=B^*\times_{\overline{\mathcal{M}}_{0, n}} \overline{\mathcal{U}}_{0, n} \ar[d]_{f|_{X^*}}\ar[r] & \overline{\mathcal{U}}_{0, n}\ar[d]\\
	B^*\ar[r] & \overline{\mathcal{M}}_{0, n}
	}
\end{equation}
Then there exists a unique flat family $\hat{f}: \hat{X}\to B$ of $n$-pointed stable rational curves satisfying the following commutative diagram
\begin{equation}\label{semistable-diagram}
	\xymatrixcolsep{5pc}\xymatrix{ X\ar[d]_f & \hat{X}=B\times_{\overline{\mathcal{M}}_{0, n}} \overline{\mathcal{U}}_{0, n}\ar@{-->}[l]\ar[d]_{\hat{f}}\ar[r] & \overline{\mathcal{U}}_{0, n}\ar[d]\\
	B & B\ar[l]_{\mbox{id}_B}\ar[r] & \overline{\mathcal{M}}_{0, n}
	}
\end{equation}	
where the broken horizontal map is a birational map such that $f^{-1}B^*\cong \hat{f}^{-1}B^*$.
\end{lemma}

\begin{proof}
Since $\overline{\mathcal{M}}_{0, n}$ is a proper scheme, by the valuative criterion of properness any morphism $B^*\to \overline{\mathcal{M}}_{0, n}$ extends uniquely to a morphism $B\to \overline{\mathcal{M}}_{0, n}$. Now since $\overline{\mathcal{M}}_{0, n}$ has a universal family $\overline{\mathcal{U}}_{0, n}$, the existence of $\hat{f}: \hat{X}\to B$ follows by taking the fiber product.	
\end{proof}
%%%%%%%%%%%%%%%%%%%%%%%%%%%%%%%%%%%%%%

%%%%%%%%%%%%%%%%%%%%%%%%%%%%%%%%%%%%%%
\begin{theorem}[Canonical Bundle Formula]\label{moduli-part}
Let $f: X\to Z$ be a proper surjective morphism, where $X$ is a normal surface and $Z$ is a smooth curve over an algebraically closed field $k$ of char $k>0$. Assume that $Q=\sum_i Q_i$ is a divisor on $Z$ such that $f$ is smooth over $(Z-\mbox{Supp}(Q))$ with fibers isomorphic to $\P^1$. Let $D=\sum_j d_jP_j$ be a $\Q$-divisor on $X$, where $d_j=0$ is allowed, which satisfies the following conditions:
\begin{enumerate}
\item $(X, D\>0)$ is KLT.
\item $D=D^h+D^v$, where $D^h=\sum_{f(D_j)=Z}d_jD_j$ and $D^v=\sum_{f(D_j)\neq Z}d_jD_j$. An irreducible component of $D^h$ (resp. $D^v$) is called horizontal (resp. vertical) component.
\item char $k=p>\frac{2}{\delta}$, where $\delta$ is the minimum non-zero coefficient of $D^h$.	
\item $K_X+D\sim_{\Q} f^*(K_Z+M)$ for some $\Q$-Cartier divisor $M$ on $Z$.
\end{enumerate}
Then there exist an effective $\Q$-divisor $D_{\div}\>0$ and a semi-ample $\Q$-divisor $D_{\mod}\>0$ on $Z$ (as defined in \ref{boundary-def}) such that
\begin{equation*}
	K_X+D\sim_\Q f^*(K_Z+D_{\div}+D_{\mod}).
	\end{equation*}
\end{theorem}

%%%%%%%%%%%%%%%%%
\begin{proof}
The sketch of the proof of this formula is given in \cite[6.7]{CTX13}. We include a complete proof following the idea of the proof of \cite[Theorem 8.1]{PS09}.\\

First we reduce the problem to the case where all components of $D^h$ are sections. Let $D_{i_0}$ be a horizontal component of $D$ and $Z'\to D_{i_0}$ be the normalization of $D_{i_0}$. Then $\nu: Z'\to Z$ is a finite surjective morphism of smooth curves. Let $X'$ be the normalization of the component of $X\times_Z Z'$ dominating $Z'$.
\begin{equation}\label{base-change-diagram}
\xymatrixcolsep{3pc}\xymatrix{ X\ar[d]_f & X'\ar[l]_{\nu'}\ar[d]_{f'}\\
Z & Z'\ar[l]_{\nu}
}
\end{equation}

Let $k=\deg(\nu:Z'\to Z)$ and $l$ be a general fiber of $f$. Then
\begin{equation}
	k=D_i\cdot l\< \frac{1}{d_i} (D\cdot l)=\frac{1}{d_i} (-K_X\cdot l)=\frac{2}{d_i}\<\frac{2}{\delta}<\chr k. 
	\end{equation}
Therefore $\nu: Z'\to Z$ is a separable morphism.\\

Let $D'$ be the log pullback of $D$ under $\nu'$, i.e.,
\begin{equation}
	K_{X'}+D'=\nu'^*(K_X+D).
	\end{equation} 
More precisely we have (by \cite[20.2]{Kol92})
\[D'=\sum_{i, j}d'_{ij}D'_{ij},\qquad  \nu'(D'_{ij})=D_i,\qquad d'_{ij}=1-(1-d_i)e_{ij},\]
where $e_{ij}$'s are the ramification indices along the $D'_{ij}$'s.\\

By construction $X$ dominates $Z$. Also, since $\nu$ is etale over a dense open subset of $Z$, say, $\nu^{-1}U\to U$, and etale morphisms are stable under base change, $(f'\circ\nu)^{-1} U\to f^{-1}U$ is etale. Thus the ramification locus $\Lambda$ of $\nu'$ does not contain any horizontal divisor of $f'$, i.e., $f'(\Lambda)\neq Z'$. Therefore $D'$ is a boundary near the generic fiber of $f'$, i.e., $D'^h$ is effective. We observe that the coefficients of $D'^h$ can be computed by intersecting with a general fiber of $f': X'\to Z'$, hence they are equal to the coefficients of $D^h\subset X$. Thus the condition $p>\frac{2}{\delta}$ remains true for $D'$ on $X'$.\\

After finitely many such base changes let $g: \tilde X\to \tilde{Z}$ be a family such that all of the horizontal components of $D_{\tilde X}$ are sections of $g$, where $D_{\tilde X}$ is the log pullback of $D$, i.e., $K_{\tilde X}+D_{\tilde X}=\psi^*(K_X+D)$.
\begin{equation}
	\xymatrixcolsep{3pc}\xymatrix{ X\ar[d]_f & \tilde{X}\ar[l]_{\psi}\ar[d]_g\\
	Z & \tilde{Z}\ar[l]_{\psi_0}
	}
	\end{equation}

By Lemma \ref{semistable-reduction}, we get a family of $n$-pointed stable rational curves $\bar{X}=\tilde{Z}\times_{\overline{\mathcal{M}}_{0, n}} \overline{\mathcal{U}}_{0, n}\to \tilde{Z}$. Let ${X'}$ be the common resolution of $\tilde X$ and $\hat{X}$. Let $\hat{X}=\tilde{Z}\times_{\overline{\mathcal{M}}_{0, n}} \mathcal{U}^*_{0, n}$. By the universal property of fiber products there exists a morphism $\mu: {X'}\to \hat{X}$. Since ${X'}$, $\tilde{X}$ and $\hat{X}$ are all isomorphic $\P^1$-bundles over a dense open subset $U\subset \tilde{Z}$, $\mu: {X'}\to \hat{X}$ is birational.
\begin{equation}\label{tilde-diagram}
\xymatrixcolsep{3pc}\xymatrix{ & &{X'}\ar@/_1pc/[dll]_\pi \ar[dl]_\lambda \ar[dr]^\mu \ar[dd]_<<<<<<<<{\tilde{f}}& &\\
X\ar[d]_f  & \tilde X\ar[l]_\psi \ar[dr]_g \ar@{-->}[rr] & & \hat{X}\ar@{-->}[r] \ar@/^2pc/[rr]^-{\hat{\phi}}\ar[dl]^{\hat{f}} & \overline{\mathcal{U}}_{0, n}\ar[r]^{\sigma}\ar[d]_{f_{0, n}} & \mathcal{U}^*_{0, n}\ar[dl]^{g_{0, n}}\\
Z & &\tilde{Z}\ar[ll]_{\psi_0}\ar[rr]^-{\phi_0} & &\overline{\mathcal{M}}_{0, n} & 
}
\end{equation}
 
Let ${D'}$ and $\hat{D}$ be $\Q$-divisors on $X'$ and $\hat{X}$ respectively, defined by
\begin{equation}\label{d-tilde}
K_{X'}+{D'}=\pi^*(K_X+D).	
\end{equation}
and 
\begin{equation*}
	K_{\hat{X}}+\hat{D}=\mu_*(K_{X'}+D').
	\end{equation*}
Since $K_{{X'}}+{D'}$ is a pullback from the base $\tilde{Z}$ (by \eqref{tilde-diagram}), by the Negativity lemma we get
\begin{equation}\label{d-hat}
	K_{{X'}}+{D'}=\mu^*(K_{\hat{X}}+\hat{D}).
	\end{equation}
Since the definition of the \emph{divisorial part} of the adjunction does not depend on the birational modification of the family (see \cite[Remark 7.3(ii)]{PS09} or \cite[Remark 3.1]{Amb99}), we will define it with respect to $\hat{f}: \hat{X}\to \tilde{Z}$. First we will show that the $\Q$-divisor $\hat{D}_{\mod}$ on $\tilde{Z}$ is semi-ample.\\ 

Since $\hat{\phi}$ is finite and $\mathcal{D}^*$ is horizontal it follows that $\hat{\phi}^*\mathcal{D}^*$ is horizontal too. Since $\hat{D}^h$ is also horizontal one sees that
\begin{equation}\label{eqn:phi-hat}
\hat{D}^h=\hat{\phi}^*\mathcal{D}^*. 
\end{equation}
From the construction of $\sigma: \overline{\mathcal{U}}_{0, n}\to \mathcal{U}^*_{0, n}$ we see that $(F, \mathcal{D}^*|_F)$ is log canonical for any fiber $F$ of $g_{0, n}: \mathcal{U}^*_{0, n}\to \overline{\mathcal{M}}_{0, n}$. Since the fibers of $\hat{f}: \hat{X}\to \tilde{Z}$ are isomorphic to the fibers of $g_{0, n}$, $(\hat{F}, \hat{D}^h|_{\hat{F}})$ is also log canonical, where $\hat{F}$ is a fiber of $\hat{f}$. Finally, since $\hat{X}$ is a surface, by inversion of adjunction $(\hat{X}, \hat{F}+\hat{D}^h)$ is log canonical near $\hat{F}$. Thus, since the fibers  of $\hat{f}$ are reduced, the $\lct$ of $(\hat{X}, \hat{D}; \hat{F})$ over the generic point of $\hat{F}$ is $(1-\mbox{coeffi.}_{\hat{F}}\hat{D})$. Hence we get
\begin{equation}\label{vertical-equation}
	\hat{D}^v=\hat{f}^*\hat{D}_{\div}.
	\end{equation}
By definition of $\hat D _{\rm mod}$ we have 
\begin{equation}\label{horizontal-equality}
	K_{\hat{X}}+\hat{D}^h\sim_\Q\hat{f}^*(K_{\tilde{Z}}+\hat{D}_{\mod}).
	\end{equation}
 
Then we have
\begin{equation}\label{horizontal-comparison}
	K_{\hat{X}}+\hat{D}^h-\hat{f}^*(K_{\tilde{Z}}+\phi_0^*\mathcal{L})=K_{\hat{X}/\tilde{Z}}+\hat{D}^h-\hat{\phi}^*K_{\mathcal{U}^*_{0, n}/\overline{\mathcal{M}}_{0, n}}-\hat{\phi}^*\mathcal{D}^*\sim_\Q 0,
	\end{equation}
where the first equality follows from  \eqref{horizontal-equality} and Lemma \ref{L-semi-ample}, and the second relation from \eqref{eqn:phi-hat} and \cite[Chapter 6, Theorem 4.9 (b) and Example 3.18]{Liu02}.\\

Since $\hat{f}$ has connected fibers, by \eqref{horizontal-equality} and \eqref{horizontal-comparison} and the projection formula for locally free sheaves, we get
\[\hat{D}_{\mod}\sim_\Q\phi^*_0\mathcal{L} \]
i.e., $\hat{D}_{\mod}$ is semi-ample.\\

Now, since $\psi _0: \tilde{Z}\to Z$ is a composition of finite morphisms of degree strictly less than $\chr k$, by \cite[Corollary 2.43]{Kol13} and \cite[Theorem 3.2]{Amb99} (also see \cite[6.6]{CTX13}) we get
\begin{equation}\label{divisorial-relation}
K_{\tilde{Z}}+\hat{D}_{\div}\sim_\Q \psi^*_0(K_Z+D_{\div}).
\end{equation}
Therefore
\begin{equation}
	\psi^*_0D_{\mod}\sim_\Q \hat{D}_{\mod}
	\end{equation}
Since $Z$ and $\tilde{Z}$ are both smooth curves, $D_{\mod}$ is semi-ample.

\end{proof}

%%%%%%%%%%%%%%%%%%%%%%%%%%%%%%%%%%%
%%%%%%%%%%%%%%%%%%%%%%%%%%%%%%%%%%%
%%%%%%%%%%%%%%%%%%%%%%%%%%%%%%%%%%%

\begin{theorem}\label{adjunction}
Let $(X, D\>0)$ be a $\Q$-factorial $3$-fold log canonical pair such that the coefficients of $D$ are contained in a DCC set $I\subset [0, 1]$. Let $W$ be a minimal log canonical center of $(X, D)$, and assume that the codimension of $W$ is $2$. Also assume that $X$ has KLT singularities and char $k>{\rm max}\{5,\frac{2}{\delta}\}$, where $\delta$ is the non-zero minimum of the set $D(I)$ (defined in \ref{dcc-def}). Then the following hold:
\begin{enumerate}
\item $W$ is normal.
\item There exists effective $\Q$-divisors $D_W$ and $M_W$ on $W$ such that $(K_X+D)|_W\sim_{\Q} K_W+D_W+M_W$. Moreover, if $D=D'+D''$ with $D'$ (resp. $D''$) the sum of all irreducible components which contain (resp. do not contain) $W$, then $M_W$ is determined only by the pair $(X, D')$.
\item There exists an effective $\Q$-divisor $M'_W$ such that $M'_W\sim_{\Q} M_W$ and the pair $(W, D_W+M'_W)$ is KLT.
\end{enumerate} 
\end{theorem}

%%%%%%%%%%%%%%%%%

\begin{proof}
Normality of $W$ follows from Theorem \ref{normality}.\\

Since $X$ is $\Q$-factorial, $(K_X+D)|_W=(K_X+D'+D'')|_W=(K_X+D')|_W+D''|_W$. Thus we may assume that all the components of $D$ contain $W$. Since $W$ is a minimal log canonical center of $(X, D)$ and $\codim_X W=2$, it does not intersect any other LC center of codimension $\>2$, by Lemma \ref{lcc-intersection}. Thus by shrinking $X$ (removing closed subsets of codimension $\>2$ which do not intersect $W$) if necessary we may assume that $W$ is the unique log canonical center of codimension $\>2$ of $(X, D)$.\\
 
Let $f: (X', D')\to (X, D)$ be a $\Q$-factorial DLT model over $(X, D)$ such that 
\begin{equation}\label{x'-log}
K_{X'}+D'=f^*(K_X+D).
\end{equation} 
Such $f$ exists by \cite[7.7]{Bir13}.\\

Note that, since $X$ is $\Q$-factorial, the exceptional locus of $f$ supports an effective $f$-anti-ample divisor. In particular all positive dimensional fibers of $f$ are contained in the support of $\lrd D'\rrd$.\\

Let $E$ be an exceptional divisor dominating $W$. Then $E$ is normal by Proposition \ref{dlt-adjunction}.  By adjunction we have 
\begin{equation}\label{E-adjunction}
	K_{E}+D'_E=(K_{X'}+D')|_{E}=f|_E^*((K_X+D)|_W)
	\end{equation}
and $(E, D'_E)$ is DLT, by Proposition \ref{dlt-adjunction} and the coefficients of $D'_E$ are in the set $D(I)$ by Lemma \ref{dcc-adjunction}. \\

By Theorem \ref{moduli-part}, there exist $\Q$-divisors $D_W=D_{\div}\>0$ and $M_W=D_{\mod}\>0$ on $W$ such that
\begin{equation}\label{W-adjunction}
	K_E+D'_E\sim_\Q f|_E^*(K_W+D_W+M_W).
	\end{equation} 
Since $f|_E: E\to W$ has connected fibers, from \eqref{E-adjunction}, \eqref{W-adjunction} and the projection formula for locally free sheaves, we get
\begin{equation}\label{final-adjunction}
(K_X+D)|_W\sim_\Q K_W+D_W+M_W.
	\end{equation}
Lemma \ref{div-independent} given below shows that $D_W$ is independent of the choice of the exceptional divisor $E$ dominating $W$.\\
 
From the definition of $D_W$ we see that $ D_W\>0$, since $D'_E\>0$. Also, since $D_W$ is independent of the birational modifications (by \cite[Remark 7.3(ii)]{PS09}) and $W$ is a minimal LC center, by taking a log resolution of $(X', D')$ and working on the strict transform of $E$, we see that the coefficients of $D_W$ are strictly less than $1$. Thus $\lrd D_W\rrd=0$.\\ 
Since $M_W$ is semi-ample and $W$ is a smooth curve, either $M_W=0$ or $M_W$ is ample. In the later case by Bertini's theorem there exists an effective $\Q$-divisor $M'_W\sim_\Q M_W$ such that $\lrd M'_W\rrd=0$ and $\supp(M'_W)\cap \supp(D_W)=\emptyset$. Hence $(W, D_W+M'_W)$ is KLT.

\end{proof}

%%%%%%%%%%%%%%%%%%
\begin{lemma}\label{div-independent}
With the same hypothesis as in Theorem \ref{adjunction}, the divisor $D_W=D_{\div}$ on $W$ is independent of the choice of the exceptional divisors dominating $W$.    
	\end{lemma}
\begin{proof}
Let $E_1$ and $E_2$ be two exceptional divisors of $f$ dominating $W$ such that 
	\begin{equation}
		K_{X'}+E_1+E_2+\Delta' =f^*(K_X+D),
	\end{equation}
	where $f: X'\to X$ is the DLT model as above and $D'=E_1+E_2+\Delta'$.\\

Notice that if $\eta_W$ is the generic point of $W$, then $f^*\eta_W\cap { \rm NKLT}(X', E_1+E_2+\Delta)$ is connected (By localizing at $\eta_W$, this follows from a surface computation involving relative Kawamata-Viehweg vanishing theorem). Therefore we may assume that $E_1\cap E_2\neq \emptyset$.\\

	By adjunction on $E_1$ we get
	\begin{equation}\label{E-equation}
	K_{E_1}+C+\Delta'_{E_1}=f|_{E_1}^*((K_X+D)|_W),	
	\end{equation}
	where $C$ is an irreducible component of $E_1\cap E_2$ dominating $W$.\\

	Adjunction on $C$ gives
	\begin{equation}
	K_C+\Delta'_C=f|_C^*((K_X+D)|_W).	
	\end{equation}

	Let $Q$ be a point on $W$, and $t=\mbox{lct}(E_1, C+\Delta'_{E_1}; f|_{E_1}^*Q)$ and $s=\mbox{lct}(C, \Delta'_C; f|_C^*Q)$. Since $C$ is an irreducible component of $E_1\cap E_2$ dominating $W$, it is enough to show that $t=s$. By adjunction, $t\< s$. So by contradiction assume that $t<s$.\\

	Since $(E_1, C+\Delta'_{E_1})$ is DLT by Proposition \ref{dlt-adjunction}, $(E_1, C+\Delta'_{E_1}+t'f|_{E_1}^*Q)$ is LC outside of $f|_{E_1}^{-1}Q$ for any $t'>t$. Thus all NLC centers of $(E_1, C+\Delta'_{E_1}+t'f|_{E_1}^*Q)$ appear along $f|_{E_1}^{-1}Q$.\\

	The general fiber of $f|_{E_1}: E_1\to W$ is isomorphic to $\P^1$. Thus $\mbox{degree}((C+\Delta'_{E_1})|_{\P^1})=2$ by \eqref{E-equation}. There are two cases depending on whether $C$ intersects the general fiber with degree $1$ or $2$.\\

	\textbf{Case I:} \emph{$C$ intersects the general fiber with degree $1$}. Then there exists a horizontal component $C'$ of $\Delta'_{E_1}$. Let $H$ be an ample divisor on $E_1$, and $F_\eta$, the generic fiber of $f|_{E_1}: E_1\to W$. Choose $\lambda>0$ such that
	\[ (H-\lambda C')\cdot F_\eta=0. \]
	Then $(H-\lambda C')|_{F_\eta}\sim_{\Q} 0$. Thus by \cite[8.3.4]{Cor07}, $H\sim_{\Q} \lambda C'-\sum\lambda_i F_i$, where the $F_i$'s are irreducible components of some fibers of $f|_{E_1}$. By adding the pullback of some appropriate divisors from the base to $\lambda C'-\sum\lambda_i F_i$, we may assume that $\lambda_i>0$ for all $i$ and $\lambda C'-\sum\lambda_i F_i$ is $f|_{E_1}$-ample.\\

	Assume that there exists a point $P\in f|_{E_1}^{-1}Q$ but $P\notin C$ such that $(E_1, C+\Delta'_{E_1}+(t+\e)f|_{E_1}^*Q)$ is not LC at $P$, where $0<\e\ll 1$ such that $t+\e<s$. Then by choosing $0<\lambda, \lambda_i\ll 1$ we can assume that $(C+\Delta'_{E_1}-\lambda C'+\sum\lambda_i F_i) \>0$, $(E_1, C+\Delta'_{E_1}-\lambda C'+\sum\lambda_i F_i+(t+\e) f|_{E_1}^*Q)$ is still not LC at $P$, and
	\begingroup
	\fontsize{10pt}{12pt}\selectfont
	\begin{equation}
		-\left(K_{E_1}+C+\Delta'_{E_1}-\lambda C'+\sum\lambda_i F_i\right)=-f|_{E_1}^*((K_X+D)|_W)+\left(\lambda C'-\sum\lambda_i F_i\right)
	\end{equation}
	\endgroup
	is $f|_{E_1}$-ample.\\

	Then by \cite[8.3]{Bir13}, $\mbox{NKLT}(E_1, C+\Delta'_{E_1}-\lambda C'+\sum\lambda_i F_i+(t+\e)f|_{E_1}^*Q)\cap f|_{E_1}^{-1}Q$ is connected. Let $R\in C\cap f|_{E_1}^{-1}Q$. Then there exists a chain of curves $G_i$'s connecting $R$ and $P$, and contained in $\mbox{NKLT}(E_1, C+\Delta'_{E_1}-\lambda C'+\sum\lambda_i F_i+(t+\e)f|_{E_1}^*Q)\cap f|_{E_1}^{-1}Q$.\\

	Now $\mbox{NKLT}(E_1, C+\Delta'_{E_1}-\lambda C'+\sum\lambda_i F_i+(t+\e)f|_{E_1}^*Q) \subset \mbox{NKLT}(E_1, C+\Delta'_{E_1}+\sum\lambda_i F_i+(t+\e)f|_{E_1}^*Q)$. Since we are only concentrating on the NKLT centers along $f|_{E_1}^{-1}Q$, we may assume that $F_i$'s are all contained in $f|_{E_1}^{-1}Q$. Then by choosing $0<\lambda_i \ll 1$ for all $i$, such that $t+\e'=t+\e+\mbox{max}\{\lambda_i\}<s$, we see that $\mbox{NKLT}(E_1, C+\Delta'_{E_1}+\sum\lambda_i F_i+(t+\e)f|_{E_1}^*Q)\subset \mbox{NKLT}(E_1, C+\Delta'_{E_1}+(t+\e')f|_{E_1}^*Q)$. Thus the curves $G_i$'s are contained in the $\mbox{NKLT}(E_1, C+\Delta'_{E_1}+(t+\e')f|_{E_1}^*Q)$. Hence $G_i$'s are contained in $\mbox{NLC}(E_1, C+\Delta'_{E_1}+sf|_{E_1}^*Q)$. This implies that $(E_1, C+\Delta'_{E_1}+sf|_{E_1}^*Q)$ is not LC at $R\in C$. Then by inversion of adjunction we get a contradiction to the fact that $(C, \Delta'_C+sf|_C^*Q)$ is LC.\\

	\textbf{Case II:} \emph{$C$ intersects the general fiber with degree $2$}. In this case $E_1\cap E_2=C$, and $\Delta'_{E_1}\text{ and }\Delta'_{E_2}$ do not have any horizontal component with respect to $f|_{E_1}\text{ and } f|_{E_2}$, respectively, where $\Delta'_{E_1}$ and $\Delta'_{E_2}$ are defined by the adjunction
	\[K_{E_1}+\Delta'_{E_1}=f|_{E_1}^*((K_X+\Delta)|_W) \quad\text{and}\quad K_{E_2}+\Delta'_{E_2}=f|_{E_2}^*((K_X+\Delta)|_W). \]
Since $D\neq 0$ and every component of $D$ contains $W$, this implies that $\Delta'$ does not contain any component of $f^{-1}_*D$ (otherwise $\Delta'_{E_i}$ will have a non zero horizontal component with respect to $f|_{E_i}$). Therefore one of the $E_i$'s must be a component of $f^{-1}_*D$, say $E_2=f^{-1}_*D_i$, where $D_i$ is an irreducible component of $D$. Thus in this case the exceptional divisors of $f$ do not intersect each other. Since $X$ is $\Q$-factorial, the exceptional locus $\mbox{Ex}(f)$ of $f:X'\to X$ supports an effective $f$-anti-ample divisor and hence $\mbox{Ex}(f)\cap f^{-1}(w)$ is connected for all $w\in W$. Thus $f$ has a unique exceptional divisor in this case and we are done.\\

	\end{proof}

%%%%%%%%%%%%%%%%%%%%%%%%%%%%%%%%%%%%%%%%%%%%%%%%%%%%%%%%%%%%%%%%%%%%%%%%%%%%%%%%%%%%%%%%%%%%%%%%%%%%%%%%%%%%%%%%%%%%%%%%%%%%%%%%%%%%%%%%%%%%%%%%%%%%%%%%%%%%%%%%%%%%%%%%%%%%%%%%%%%%%%%%%%%%%%%%%%%%%%%%%%%%%%%%%%%%%%%%%%%%%%%%%%%%%%%%%%%%%%%%%%%%%%%%%%%%%%%%%%%%%%%%%%%%%%%%%%%%%%%%%%%%%%%%%%%%%%%%%%%%%%%%%%%%%%%%%%%%%%%%%%%%%%%%%

\bibliographystyle{hep}
\bibliography{references.bib}	
\end{document}